\documentclass[conference]{IEEEtran}
\IEEEoverridecommandlockouts
\usepackage{cite}
\usepackage{amsmath,amssymb,amsfonts}
\usepackage{graphicx}
\usepackage{textcomp}
\usepackage{xcolor}
\usepackage{algorithm,algorithmicx,algpseudocode}
\usepackage{mathtools}
\usepackage{amsthm}
\usepackage{placeins}
\usepackage{booktabs}
\usepackage{tikz}
\usetikzlibrary{positioning, shapes.geometric}
\algdef{SE}[DOWHILE]{Do}{doWhile}{\textbf{Do}}[1]{\textbf{While}\ #1}
\algdef{SE}{Begin}{End}{\textbf{Begin}}{\textbf{End}}
\newcommand{\norm}[1]{\left\lVert#1\right\rVert}
\allowdisplaybreaks[4]

\newtheorem{assumption}{\textbf{Assumption}}

\newtheorem{theorem}{\textbf{Theorem}}
\newtheorem{lemma}{\textbf{Lemma}}
\newtheorem{condition}{\textbf{Condition}}
\renewenvironment{proof}{{\bfseries Proof. }}{\qed}
\renewenvironment{proof}[1][\proofname]{{\bfseries #1. }}{\qed}

\def\BibTeX{{\rm B\kern-.05em{\sc i\kern-.025em b}\kern-.08em
    T\kern-.1667em\lower.7ex\hbox{E}\kern-.125emX}}
\begin{document}
\title{Strategic Investments of Large Scale Battery Energy Storage Systems in the Wholesale Electricity Market}

\author{Ang Li;
		Yubo Wang,~\IEEEmembership{Senior Member,~IEEE;}
        Lei Fan,~\IEEEmembership{Senior Member,~IEEE;}
        Jiming Peng

\thanks{Ang Li is with the Department of Industrial Engineering at University of Houston, Houston, TX, E-mail: ali21@central.uh.edu, and is with Siemens Technology, NJ, USA, Jiming Peng is with the Department of Industrial Engineering at University of Houston, Houston, TX, E-mail: jopeng@central.uh.edu, Yubo Wang is with the Siemens Technology, NJ, USA, E-mail: ybwang25@gmail.com, Lei Fan is with the Department of Engineering Technology at University of Houston, Houston, TX, E-mail: lfan8@uh.edu.}
}

\maketitle

\begin{abstract}

In this paper, we study the strategic investment problem of  battery energy storage systems (BESSs) in the wholesale electricity market from the perspective of BESSs owners. Large-scale BESSs planning without considering the possible wholesale market price change may result in possible locational marginal price (LMP) changes. To overcome such limits, we propose a three-phase approach for the BESS investment problem. In Phase-1, we conduct a search for the optimal BESS configurations via a congestion-based heuristics and Bayesian optimization. In Phases 2 and 3, we alternatively dispatch optimization tasks to optimize the wholesale market clearing for LMPs and identify the optimal schedule for BESSs operations. We validate the model using a ten-year simulation on the Electric Reliability Council of Texas (ERCOT) market. Experimental results show that the Bayesian optimization model runs 16 times faster than the grid search model in achieving the same solution quality. Further, the iterative method demonstrates that the model without considering possible LMP changes makes a 21\% profit overestimation. 

\end{abstract}


\def\Tt{\mathcal{T}}

\section*{Nomenclature}
\addcontentsline{toc}{section}{Nomenclature}
\begin{IEEEdescription}[\IEEEusemathlabelsep\IEEEsetlabelwidth{$V_1,V_2$}]
	\item [A. Sets]
	\item[$\Tt$] The index set of operation time span.
	\item[$I$] The index set of candidate batteries. 
	\item[$L$] The index set of transmission lines. 
	\item[$M$] The index set of traditional generating units. 
	\item[$B$] The index set of buses. 
	\item [B. Parameters]
	\item[$T$] The cardinality of $ \Tt $. 
	\item[$F_i$] The fixed cost of installing battery $ i $. 
	\item[$G_i$] The cost of battery $ i $ per unit of capacity. 
	\item[$\beta_m$] The generation cost of generator $ m $. 
	\item[$FL_l$] The transmission capacity on transmission line $ l $. 
	\item[$LD_{bt}$] The load of bus $ b $ at time period $ t $. 
	\item[$SF$] The power transfer distribution factors (PTDF), also known as the shift factor. 
	\item[$A$] The total budget. 
	\item[$\kappa^{c}$] The maximum charging rate coefficient for battery $ i $. 
	\item[$\kappa^{d}$] The maximum discharging rate coefficient for battery $ i $. 
	\item[$S^{l}$] The minimum state of charge (SOC) for the battery energy storage.
	\item[$S^{u}$] The maximum state of charge (SOC) for the battery energy storage.
	\item[$P^{l}_{m}$] The minimum power generation of generator $ m $.
	\item[$P^{u}_{m}$] The maximum power generation of generator $ m $.
	\item[$\lambda_t$] Electricity prices at period $t$.
	\item[$\eta^{c}$] The charging efficiency.
	\item[$\eta^{d}$] The discharging efficiency.
	\item[$\Delta$] The length of time interval. 
	\item [C. Decision Variables]
	\item[$p_{it}^{c}$] The charging power of battery $ i $ at period $t$.
	\item[$p_{it}^{d}$] The discharging power of battery $ i $ at period $t$. 
	\item[$p_{mt}$] The power generation of generator $ m $ at period $t$. 
	\item[$\alpha_{it}^{c}$] The charging bid of battery $ i $ at period $t$.
	\item[$\alpha_{it}^{d}$] The discharging bid of battery $ i $ at period $t$.
	\item[$c_i$] The capacity of battery $ i $.
	\item[$y_i$] The binary indicator of installing battery $ i $ or not. 
	\item[$e_{it}$] The stored energy of battery $ i $ at time period $t$. 
\end{IEEEdescription}

\section{Introduction}\label{sec:intro}

In recent years, large amounts of renewable energy resources are integrated into the electrical grid, changing the operation and control of the electrical grid drastically~\cite{jiang2011robust}. 
Traditional decision methods and approaches cannot deal with the high volatility and uncertainty of these renewable energy resources~\cite{jiang2011robust}.
The installation of the energy storage systems (ESSs) in electrical grid can improve the flexibility of the power grid, which further provides more options for system operators to deal with the uncertainties from the non-dispatchable renewable energy resources~\cite{ribeiro2001energy}. 
Battery energy storage systems (BESSs) become one of the most popular technologies in the ESSs because the BESSs have the characteristics such as high power and energy densities, and short lead time of installation~\cite{chen2009progress}. 

From the BESSs owner's perspective, it is crucial 
to first identify the optimal installation location of the BESSs and the optimal BESSs size, and then identify the optimal operational schedule for the BESS based on locational marginal prices. 
Numerous researchers have explored BESSs location problem in the literature. 
Researchers in~\cite{zhao2015review} reviewed the wind power integration and ESS siting and sizing. 
Researchers in~\cite{zidar2016review} reviewed researches on ESS sizing and siting, and discussed challenges of integrating ESS into distribution networks. 
Researchers in~\cite{celli2009optimal} studied the distribution network planning including ESSs by applying the genetic algorithm and dynamic programming. 
Researchers in~\cite{padhee2019fixed} studied BESS allocation problem for transmission network under uncertainties. 
The authors in~\cite{li2016optimal} studied the coupling of large-scale BESS and renewables such as wind farms and photovoltaic power stations in the transmission grid. 
The authors in~\cite{chen2020optimization} proposed a probabilistic non-linear optimization problem for BESS size and location. 
The authors in~\cite{talluri2021optimal} proposed a BESS scheduling strategy in a renewable energy community with methods including a machine learning-based forecast algorithm, mixed integer linear programming, and decision tree algorithm. 
The authors in~\cite{karanki2013optimal} proposed a loss sensitivity based algorithm to determine the optimal site and size of BESS in the distribution system, analyzing the influence of BESS allocation and sizing on system characteristics. 
The authors in~\cite{pandvzic2014near} studied the optimal siting and sizing of BESSs to reduce transmission network congestion and enable large-scale integration of renewable energy resources. 
Note that the installation of   large-scale BESSs in power grids can have a significant impact on the  pricing of the electricity in the  grid \cite{pandvzic2014near}. 
However, most existing works in BESSs siting and sizing assume the fixed nodal price in the wholesale market and ignore the possible nodal price changes 
caused by the installation of   the BESSs. 

In this paper, we focus on the BESSIP in the electricity market from the perspective of the BESSs owners, such as independent power producers (IPPs) and utilities~\cite{wankmuller2017impact,krishnamurthy2017energy,fang2018mean}.
Particularly, we take  the impact on the electricity price in the electrical grid from  large-scale BESSs with different sizes and installation locations into account, and investigate how the installation of BESSs may affect the BESSs' investment return and operational costs. 
The BESS sizing and siting problem solution can help IPPs and utilities with their strategic planning in the wholesale markets. In addition, the long-term simulation for the BESS sizing and siting problem can help the IPPs and utilities analyze the BESS's investment return. 

Efforts have been made by Independent System Operators (ISOs) to develop regulations and policies for the integration of BESSs into the electrical grid. 
Due to the distinctive operating characteristics of BESSs, this integration involves both the ISOs and the BESSs owners. 
The ISOs aim to minimize the total system cost, maximize system reliability, and reduce system emissions. Meanwhile, BESSs owners aim to maximize their investment returns while minimizing operational costs. 
Pareto optimality is a concept widely used in economics and engineering, which refers to a state of resource allocation where it is impossible to improve one agent's objective function without negatively impacting another agent's. 
Many researchers have applied the Pareto optimality to the study of power system operation and planning. 
Researchers in \cite{ngatchou2005pareto} discussed metaheuristic approaches for solving multi-objective optimization problems. 
Researchers in \cite{yassami2010power} presented a new method for designing power system stabilizers using the Strength Pareto approach and compared it with genetic algorithms. 
Researchers in \cite{nikmanesh2016pareto} studied a uniform-diversity genetic algorithm to optimize PI/PID controllers in Load Frequency Control of power systems for Pareto optimization. 
Researchers in \cite{taheri2017economic} proposed a particle swarm optimization algorithm to solve the economic dispatch problem in power systems while considering environmental pollution. 
Researchers in \cite{wang2019multi} proposed a wind and photovoltaic power BESS robust scheduling model that considers the time-of-use price to reduce the influence of uncertainty in power output. 
Researchers in \cite{tan2015pareto} focused on optimizing the operation of a distributed BESS for energy arbitrage under dynamic pricing, taking into account the economic value and lifetime tradeoff of the BESS. Researchers in \cite{tan2015pareto} proved the optimal policy generated by a parallel algorithm is Pareto optimal. 
Researchers in \cite{wu2019vsc} proposed a new model for voltage source converter based BESSs that interface with power grids, incorporated into an active-reactive optimal power flow problem.  Researchers in \cite{wu2019vsc} developed a sequence of strong relaxations to transform the proposed problem into a mixed-integer second-order cone programming problem. 
Researchers in \cite{al2020reinforcement} proposed a cooperative SOC control scheme to alleviate capacity problems of BESSs used to mitigate over-voltage issues caused by high levels of distributed photovoltaic generation. 
New methodologies are required to address the computational costs of the planning models, theoretical explorations are also needed to understand the impact of BESSs on the electricity market.
The strategic planning of BESSs usually requires intensive computations. To reduce the computational workload, most existing studies on BESS sizing and siting problem simulated the BESSs' operations in a short time span and a relatively small distribution network or transmission network~\cite{celli2009optimal,padhee2019fixed,pandvzic2014near}. 
To address the computational challenge in large scale system, we propose to adopt  the Sequential Model-Based Optimization (SMBO) for parameter search developed  in~\cite{hutter2011sequential} to solve the siting and sizing optimization model in Phase 1. SMBO has been widely used in machine learning and hyperparameter optimization thanks to its efficiency and versatility~\cite{feurer2019hyperparameter}. It looks for new most promising parameters to refine the model in an iterative manner. For more discussion on SMBO and Bayesian optimization, we refer to \cite{snoek2012practical} and the references therein. 
We refer to~\cite{bergstra2011algorithms,thornton2013auto} for more details on SMBO and Bayesian optimization. 
In this paper, we propose a Bayesian optimization framework for BESS sizing and siting problem, which enables the BESS sizing and siting problem to be solved in a large distribution network or transmission network with thousands of buses and transmission lines. 
We also propose a heuristic method to effectively reduce the computational complexity of the Bayesian optimization framework.

Our main contributions are summarized as follows:

\begin{itemize}
\item We propose a three-phase approach for the BESSIP on large scale. The proposed approach  solves several   sub-problems in different phases, and uses the solutions from the sub-problems   to tune the parameters in the optimization models in these phases to better characterize the interaction between ISO and the BESS owners.  
\item We propose an effective heuristic method  to effectively solve the  siting and sizing problem to help the BESS owners to select the locations and sizes of the BESSs in Phase 1, and integrate it into the Bayesian optimization framework to further refine the underlying optimization model.
\item We studied the existence of Pareto optimality in the BESSIP model, and show that the optimal solution is always Pareto optimal. 
\item We propose an alternative iterative approach for a two-agent model to  accurately estimate the profit after installation of the BESSs. We also establish the convergence of the new algorithm. 
\end{itemize}

This paper is organized as follows. 
In Section~\ref{sec:formulation}, we describe the BESSIP and discuss the challenges and limits in it.  
In Section~\ref{sec:hyperopt},  we describe the three-phase approach for the BESSIP, and discuss how to effectively solve the subproblems in each phase. 
In Section~\ref{sec:exp}, we show the numerical experiments and case studies.
In the end, we conclude the paper in Section~\ref{sec:conclusion}.

\section{BESS Investment Model}\label{sec:formulation}

In this section, we introduce a new BESS investment problem  for large-scale BESSs (BESSIP)  based on a centralized planning model considering the economic dispatch with BESSs, which can be formulated as a two-agent model. BESSIP also takes the choices of the locations and sizes of the BESSs into account. Then we present variants of the BESSIP and study their properties.

\subsection{Two-Agent Model}\label{subsec:two-agent}
In this subsection, we propose a two-agent optimization model based on electricity market operation process in Electric Reliability Council of Texas (ERCOT). The first model is the economic dispatch model used by ISO. The second is the battery owners' self-scheduling model. 

\begin{subequations}\label{model:eco}
	\begin{align}
		&\min \sum_{i \in I}(\sum_{t=1}^{T}\alpha^c_{it}p_{it}^c+\sum_{t=1}^{T}\alpha^d_{it}p_{it}^d ) +\sum_{m \in M}\sum_{t=1}^{T}\beta_m p_{mt}\\
		&\mbox{s.t.\quad}\sum_{i \in I}p^d_{it}+\sum_{m \in M}p_{mt}=LD_t+\sum_{i \in I}p^c_{it}; \label{con:BESS-SS-balance}\\
		&\phantom{\mbox{s.t.\quad}} -FL_l\le \sum_{b \in B}SF_{lb} (\sum_{i \in I_b}p^d_{it}+\sum_{m \in M_b}p_{mt}  \nonumber \\
		&\phantom{\mbox{s.t.\quad}}\phantom{\mbox{s.t.\quad}}-LD_{bt}-\sum_{i \in I_b}p^c_{it}) \le FL_l, \forall l \in L; \label{con:BESS-SS-line}\\
		&\phantom{\mbox{s.t.\quad}} LD_t=\sum_{b \in B}LD_{bt}; \label{con:BESS-SS-demand}\\
		&\phantom{\mbox{s.t.\quad}} P_m^l \le p_{mt} \le P_m^u, \forall m \in M; \label{con:BESS-SS-box-pm}\\
		&\phantom{\mbox{s.t.\quad}} P^{cl}_{it} \le p^c_{it} \le P^{cu}_{it}, \forall i \in I, \forall t \in \Tt; \label{con:box-pc} \\
		&\phantom{\mbox{s.t.\quad}} P^{dl}_{it} \le p^d_{it} \le P^{du}_{it}, \forall i \in I, \forall t \in \Tt. \label{con:box-pd} 
	\end{align}
\end{subequations}
In model (\ref{model:eco}), the objective  is to  minimize the total cost  including the BESSs charging/discharging cost and thermal generation cost. (\ref{con:BESS-SS-balance}) represents the power balance constraints. (\ref{con:BESS-SS-line}) represents the power flow limits. (\ref{con:BESS-SS-demand}) represents the total demand. (\ref{con:BESS-SS-box-pm}) represents the minimum and maximum generation rates of thermal generators. 
(\ref{con:box-pc}), (\ref{con:box-pd}) represent the minimum and maximum charging/discharging rates of BESSs.
From the BESSs owner's perspective, a complete bidding strategy includes not only the biddings $ \alpha_{it}^c, \alpha_{it}^d $, but also lower and upper bounds for charging/discharging rates. 
Therefore, in the economic dispatch model~(\ref{model:eco}), we generalize the bidding strategy of the BESSs owner by adding constraints~(\ref{con:box-pc}), (\ref{con:box-pd}). 

Let $ \mu $, $ \pi^+ $, $ \pi^- $ be the Lagrangian multipliers to constraints~(\ref{con:BESS-SS-balance}), (\ref{con:BESS-SS-line}). The locational marginal price (LMP) on bus $ b\in B $ at time $ t\in \Tt $, denoted by $ \lambda $, is defined as
\begin{align}\label{def:lmp}
	\lambda = \mu + \sum_{l \in L}SF_{lb} (\pi^+_l-\pi^-_l). 
\end{align}

Once  the LMP $ \lambda $ is obtained based on the optimal solution of model~(\ref{model:eco}) and definition~(\ref{def:lmp}), the BESS owner needs to solve the following  scheduling problem  to determine the bidding strategy. 
\begin{subequations}\label{model:bat}
	\begin{align}
		&\min  F_{i}y_{i} + G_{i}c_{i} + \sum_{t=1}^{T}\lambda_t (p_{it}^c-p_{it}^d)  \label{obj:bat} \\
		&\mbox{s.t.\quad} 0\le p_{it}^{c} \leq \kappa^c c_{i}, \forall i \in I,  \forall t \in \Tt; \label{con:BESS-SS-pc} \\
		&\phantom{\mbox{s.t.\quad}} 0\le p_{it}^{d} \leq \kappa^d c_{i}, \forall i \in I,  \forall t \in \Tt;  \label{con:BESS-SS-pd} \\
		&\phantom{\mbox{s.t.\quad}} S^l c_{i} \leq e_{it} \leq S^u c_{i}, \forall i \in I,  \forall t \in \Tt; \label{con:BESS-SS-box-s}\\
		&\phantom{\mbox{s.t.\quad}} e_{i,t+1} = e_{it}+ p^{c}_{it}\eta^{c}\Delta\! -\!\frac{p^{d}_{it}}{\eta^{d}} \Delta, \forall i \in I,  \forall t \in \Tt;\label{con:BESS-SS-soc} \\
		&\phantom{\mbox{s.t.\quad}}c_i \le M y_{i}, \forall i \in I; \label{con:BESS-SS-int-M}\\
		&\phantom{\mbox{s.t.\quad}}	\sum_{i \in I}F_{i}y_{i} + \sum_{i \in I}G_{i}c_{i} \le A; \label{con:BESS-SS-budget}\\
		&\phantom{\mbox{s.t.\quad}} y_{i} \in \{0,1\}. \label{con:BESS-SS-int}
	\end{align}
\end{subequations}

In model (\ref{model:bat}), the objective  is to  minimize the total cost  including the BESSs charging/discharging cost, fixed and per unit of capacity cost. 
(\ref{con:BESS-SS-pc}), (\ref{con:BESS-SS-pd}) represent the minimum and maximum charging/discharging rates of BESSs, where $ \kappa^c, \kappa^d $ are power rating coefficients of BESSs. (\ref{con:BESS-SS-box-s}) represents the minimum and maximum energy storage of BESSs. (\ref{con:BESS-SS-soc}) represents the state of charge (SOC) of BESSs, where $ \eta^c, \eta^d $ are charging/discharging efficiencies and $ \Delta $ is the length of a time interval. (\ref{con:BESS-SS-budget}) represents the budget constraint.

The BESSs owner chooses a univalued bidding strategy function $ \mathcal{S} $ based on the regulations and specific trading strategies for arbitrage. 
The bidding strategy function $ \mathcal{S} $ takes solution to model~(\ref{model:bat}) and the LMP $ \lambda $ as inputs and generates the biddings. 
Specifically, after solving model~(\ref{model:bat}), the BESS owner can obtain their bidding strategy $ (\alpha^c_{it}, \alpha^d_{it}, P^{cl}_{it}, P^{cu}_{it}, P^{dl}_{it}, P^{du}_{it}) = \mathcal{S}(p_{i}^{c*},p_{i}^{d*},\lambda) $ from the optimal solution and the LMP. 
$ \alpha^c_{it}, \alpha^d_{it} $ are the price biddings generated by the BESSs owner, and $ P^{cl}_{it}, P^{cu}_{it}, P^{dl}_{it}, P^{du}_{it} $ are the lower and upper bounds for charging/discharging rates of the bidding. 
In ERCOT system, the BESSs owner only reports the biddings to the ERCOT market operator, and the BESSs status are defined as private information in the ERCOT market. 
The power system operation can be significantly influenced by the bidding strategy $\mathcal{S}$. Moreover, the BESSs owner needs to choose the bidding strategy function $\mathcal{S}$ carefully to maximize the profit while bidding into the ERCOT market, as the bidding prices cannot be higher than LMP for discharging and lower than LMP for charging, otherwise the BESSs owner cannot bid into the market and the loss of profit is inevitable. 
In this paper, we will examine how different bidding strategies affect the power market. Specifically, we will conduct a theoretical analysis to determine the circumstances in which the bidding strategies can ensure Pareto efficiency in the power market.

The two-agent BESSIP model~(\ref{model:eco})-(\ref{model:bat}) is a mixed-integer optimization model in large scale, and it cannot be solved efficiently by commercial servers and solvers.  
To address these issues, we propose a new three-phase optimization framework which can better reflect the interaction between the system operator and battery owners in the market. 
The three-phase optimization framework regards the size and location of the BESSs as hyperparameters of the optimization problem, since the selection of the locations and sizes of the BESSs is strategically decided by the BESS owners in advance.

\subsection{Variants of BESSIP Model}
The literature's first variant presumes that the fixed cost of $ F_i = 0 $. 
This variant is based on the fact that the fixed cost of installing BESS is typically negligible when compared to the overall cost of the BESSs. This is because BESSs are well known for its flexibility, adaptability and scalability in providing a range of services, at the cost of relatively higher per unit cost, compared to other energy storage technologies such as pumped hydroelectric storage. With fixed cost $ F_i = 0 $, the BESSIP model can eliminate its integer variables, resulting in the following model for the first variant.

\begin{subequations}\label{model:S}
	\begin{align}
		&\min G_{i}c_{i} + \sum_{t=1}^{T}\lambda_t (p_{it}^c-p_{it}^d) \\
		&\mbox{s.t.\quad} \mbox{Constraints~(\ref{con:BESS-SS-pc})-(\ref{con:BESS-SS-soc})}; \\
		&\phantom{\mbox{s.t.\quad}}	\sum_{i \in I}G_{i}c_{i} \le A. \label{con:S-budget}
	\end{align}
\end{subequations}

In the BESSIP variant model, the decision variable $c_i$ represents the capacity of the BESS at location $i$, and $c_i > 0$ if and only if BESS is installed at location $i$. 
$c_i$ can be categorized as hyperparameters in the relaxed BESSIP model, as the impact of $c_i$ on the objective function is difficult to quantify without solving the model itself. 
By fixing $c_i$, the BESSIP model can be reduced to a standard economic dispatch model considering the BESS. 

Another variant from the literature adds a complimentarity constraint as follows in the optimization problem. 
\begin{align}
	 p^{c}_{it}p^{d}_{it} = 0, \forall i \in I, \forall t \in \Tt; \label{con:complimentary}
\end{align}
This constraint reflects the inherent limitation that BESSs cannot charge and discharge simultaneously. Although this variant can enhance the BESSIP model's accuracy, it adds a bilinear constraint that is nonlinear and nonconvex, thereby increasing computational complexity.
We will conduct numerical experiments to compare the solution quality and computational costs of these variants of BESSIP model.

\section{Three-Phase Optimization Framework }\label{sec:hyperopt}
This section consists of four subsections. 
In Subsection~\ref{subsec:heuristic},  we propose a heuristics to solve the siting and sizing problem  in Phase 1 to help the BESS owners select the locations and sizes of the BESSs. 
In Subsection~\ref{subsec:three-phase}, we describe the integrated three-phase approach.  
Finally, in Subsection~\ref{subsec:convergence}, we discuss the convergence of the three-phase approach.

\subsection{A Heuristics for the BESS Siting and Sizing Problem}\label{subsec:heuristic}

In this subsection, we introduce a new heuristic method to narrow down the search space for promising BESSs location candidates and their sizes. 

We first observe that by installing BESSs in the electrical grid, the fluctuation of the supply and demand mismatch is shaved because BESSs can charge when the demand is low and discharge when the demand is high. On the other hand, BESSs can arbitrage chronologically by storing unused energy over low-price time periods. The price dynamics of the electricity market, however, are dominated by the marginal congestion cost component~(\ref{def:cong}) of the dual optimization problem. 

\begin{align}\label{def:cong}
	\sum_{l \in L}SF_{lb} (\pi^+_l-\pi^-_l). 
\end{align}

We remark that from duality theory, if $ \pi^+_l>0 $, the upper bound inequality of line $ l $ is active, and because any power injection to the buses corresponding to the positive component of PTDF will further increase the left hand side of the upper bound inequality in~(\ref{con:BESS-SS-line}), the multiplication of $ \pi^+_l $ and $ SF^+_{lb} $ maps the line congestion to the bus power injection. A similar conclusion can be drawn for the negative components. 

Hence, let $ \pi^+ $, $ \pi^- $ be the Lagrangian multipliers to constraints~(\ref{con:BESS-SS-line}).
Let $ SF^{+} $ and $ SF^{-} $ be the positive and negative components of the PTDF $ SF $ matrix respectively. 
We define the score of congestion on the buses as follows: 
\begin{align}
	S_{cong}=\mathbb{E}_{\Tt}[\norm{\sum_{l \in L}SF_{lb}^{+}\pi^+_l}_p + \norm{\sum_{l \in L}SF_{lb}^{-}\pi^{-}_l}_p]. 
\end{align}
In this paper, we use the following $ L_1 $ norm. 
\begin{align}
	S_{cong}=\mathbb{E}_{\Tt}[-\sum_{l \in L}(SF_{lb}^{+}\pi^+_l - SF_{lb}^{-}\pi^{-}_l)]. 
\end{align}

According to heuristics, BESSs make profits by power arbitrage in the wholesale electricity market. A higher score of congestion reflects higher price volatility, revealing the potential for the BESSs to have higher profits. Therefore, we narrow down our selection of candidate BESSs location to those with the highest $ S_{cong} $. 
$ S_{cong} $ can be estimated by simulation on the basis of historical data. In this paper, we calculate $ S_{cong} $ of past 5 years of ERCOT data, then we use $ S_{cong} $ as a heuristic criterion to generate candidate BESS installation configurations. Initially, each location is selected with a probability proportional to its $ S_{cong} $. In subsection~\ref{subsec:three-phase}, we propose a Bayesian optimization method to update the generative model and further refine the candidate BESS installation configurations.

\subsection{Integrated Three-Phase Approach}\label{subsec:three-phase}

In this subsection, we describe an integrated three-phase  approach for the BESSIP model, whose structure is shown in Figure~\ref{fig:3levels}. The proposed framework has three phases.  

In Phase 1 we try to find the most promising location and size of the BESSs in a pool of feasible sizes and locations using Bayesian optimization method\cite{hutter2011sequential} as shown in Algorithm~\ref{alg:bayes}. 

Based on the selected locations and sizes  (denoted by $ c_i $) in Phase 1, in Phase 2 we will simulate the electricity market operations in a long planning horizon (e.g., 10 years) and calculate the net present value (NPV) of return expectation with this location and size choice. Each iteration of the simulation will cover a 24 hours or 48 hours period. Then the simulation will rolling solve the whole study periods. 

In each iteration of the market simulation,  two optimization models will be solved in an iterative manner in Phase 3.  More specifically, In each iteration of the simulation in Phase 3, one agent represents the ISO to run the economic dispatch model (as in model (\ref{model:eco})) with the bids from battery owners and publish the LMPs. Then the battery owners run the self-scheduling model (as in Model (\ref{model:bat})) to find the best schedule and  the corresponding bidding strategy. Then the ISO agent will take the bids and re-solve the economic dispatch model (as in model (\ref{model:eco})) to find a new market clearing price. we repeat such a process  until convergence is reached. The detailed algorithm is shown in Algorithm~\ref{alg:aus}. The proof of convergence is given in the next subsection. 

In other words, we break down the whole optimization process into three phases. In each phase, we   deal with some reduced sub-problems that are relatively   easy to solve. Such an approach allows us to effectively obtain a good solution to the extremely challenging BESSIP in large-scale with limited  computational resources. 

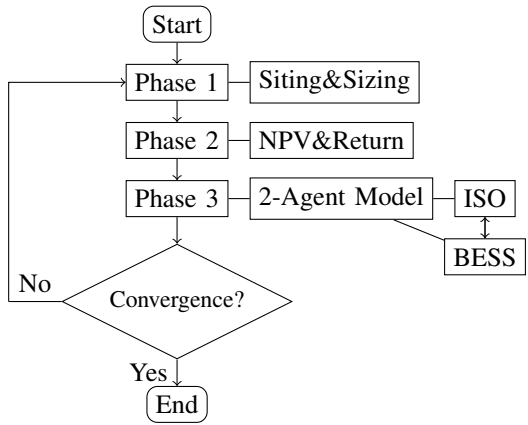
\begin{figure}[htbp]
	\centering
	\begin{tikzpicture}[node distance=8pt]
		\node[draw, rounded corners]                        (start)   {Start};
		\node[draw, below=of start]                         (level 1)  {Phase 1};
		\node[draw, right = of level 1]                      (hyp)  {Siting\&Sizing};
		\node[draw, below=of level 1]                        (level 2)  {Phase 2};
		\node[draw, right = of level 2]                      (dp)  {NPV\&Return};
		\node[draw, below=of level 2]                        (level 3)  {Phase 3};
		\node[draw, right = of level 3]                      (two)  {2-Agent Model};
		\node[draw, right = of two]                      (iso)  {ISO};
		\node[draw, below = of iso]                      (bess)  {BESS};
		\node[draw, diamond, aspect=2, below= 10pt of level 3, align=center]     (choice)  {\small Convergence?};
		\coordinate[left = 20pt of choice] (loop);
		\node[draw, rounded corners, below=10pt of choice]  (end)     {End};
		
		\draw[->] (start)  -- (level 1);
		\draw[->] (level 1) -- (level 2);
		\draw[->] (level 2) -- (level 3);
		\draw[->] (level 3) -- (choice);
		\draw[->] (choice) -- node[left]  {Yes} (end);
		\draw[-] (choice) -- node[above] {No}  (loop);
		\draw[->] (loop) -- (loop|-level 1) -> (level 1);
		\draw[-] (level 1)  -- (hyp);
		\draw[-] (level 2)  -- (dp);
		\draw[-] (level 3)  -- (two);
		\draw[->] (iso) -- (bess);
		\draw[->] (bess) -- (iso);
		\draw[-] (two)  -- (iso);
		\draw[-] (two)  -- (bess);
	\end{tikzpicture}
	\caption{Structure of BESSIP Optimization Framework}
	\label{fig:3levels}
\end{figure}

\begin{algorithm}
	\caption{Bayesian Optimization Framework }\label{alg:bayes}
	\begin{algorithmic}[0]
		\Begin
		\State Set $k=0$.
		\State Initialize location set $ I^0 $, capacities $ c_i^0 $ with Bayesian optimization, search history $ \mathcal{H} = \emptyset $. 
		\While{Computational resources not exhausted}
		\State Run Algorithm~\ref{alg:aus} with parameters $ I^k $, $ c_i^k $ to calculate expectation of return $ R $ and updated $ S_{cong} $. 
		\State $ \mathcal{H} = \mathcal{H} \cup \{(R,S_{cong},I^k,c_i^k)\} $. 
		\State Use Bayesian optimization to update probabilities and generate new most promising parameters $ I^{k+1} $, $ c_i^{k+1} $. 
		\State $ k = k+1 $. 
		\EndWhile
		\State Output search history $ \mathcal{H} $ and optimal parameters $ I^* $, $ c_i^* $. 
		\End
	\end{algorithmic}
\end{algorithm}

\begin{algorithm}
	\caption{Alternative Updating Scheme (AUS)}\label{alg:aus}
	\begin{algorithmic}[0]
		\Begin
		\State Set $k=0$.
		\State Initialize the trading strategy $ \mathcal{S}^0 $. 
		\While{Computational resources not exhausted \&\& $ k < \mbox{MAX\_ITER} $}
		\State Solve model~(\ref{model:eco}) with $ \mathcal{S}^k $ for $ \lambda^k $. 
		\State Solve model~(\ref{model:bat}) with $ \lambda^k $ for $ \mathcal{S}^{k+1} $. 
		\State $ k = k+1 $. 
		\If{$ k!= 1 $ \&\& $ \norm{\lambda^{k}-\lambda^{k-1}} < \epsilon $}
		\State Break. 
		\EndIf
		\EndWhile
		\State Output current value of the investment. 
		\End
	\end{algorithmic}
\end{algorithm}

\subsection{Sensitivity Analysis of Economic Dispatch}\label{subsec:sensitivity}

In this subsection, we present the sensitivity analysis of the economic dispatch model~(\ref{model:eco}) with respect to the bidding parameters. 
Let $ \mu $, $ \pi^+ $ and $ \pi^- $ be the Lagrangian multipliers to constraints~(\ref{con:BESS-SS-balance}), (\ref{con:BESS-SS-line}). Let $ \gamma^c, \gamma^d, \gamma $ be the Lagrangian multipliers to the box constraints. Then we can write the dual problem to model~(\ref{model:eco}) as follows. 

\begin{subequations}
	\begin{align}
		&\max \quad  (\sum_{b \in B}LD_{b})\mu + \sum_{l \in L}FL_l (\pi^+_l + \pi^-_l) \\
		&+ \sum_{i \in I}P^{cu}_i \gamma^{c+}_i -\sum_{i \in I}P^{cl}_i \gamma^{c-}_i \\
		&+ \sum_{i \in I}P^{du}_i \gamma^{d+}_i -\sum_{i \in I}P^{dl}_i \gamma^{d-}_i \\
		&+ \sum_{m \in  M}P^{u}_m \gamma^{+}_m -\sum_{m \in M}P^{l}_m \gamma^{-}_m  \\
		&\mbox{s.t.\quad}\mu + \sum_{l \in L}SF_{lb(i)} (\pi^+_l-\pi^-_l) + (\gamma^{d+}_i-\gamma^{d-}_i) = \alpha^d_i, \forall i \in I;  \\
		&\phantom{\mbox{s.t.\quad}} \mu + \sum_{l \in L}SF_{lb(i)} (\pi^+_l-\pi^-_l) + (\gamma^{c+}_i-\gamma^{c-}_i) = -\alpha^c_i, \forall i \in I;\\
		&\phantom{\mbox{s.t.\quad}} \mu + \sum_{l \in L}SF_{lb(m)} (\pi^+_l-\pi^-_l) + (\gamma^{+}_m-\gamma^{-}_m) = \beta_m, \forall m \in M; \\
		&\phantom{\mbox{s.t.\quad}} \mu  \text{\quad unrestricted}, \pi^{\pm}, \gamma^{c\pm}, \gamma^{d\pm}, \gamma^{\pm} \le 0. 
	\end{align}
\end{subequations}

Recall the definition of LMP of equation (\ref{def:lmp}), from the sensitivity analysis theory, we have the following theorem. 

\begin{theorem}\label{thm:sensitivity}
	There exists $ \alpha_{it}^{c,l} $, $ \alpha_{it}^{c,u} $, $ \alpha_{it}^{d,l} $, $ \alpha_{it}^{d,u} $, such that if $ \alpha_{it}^c \in [\alpha_{it}^{c,l}, \alpha_{it}^{c,u}] $, $ \alpha_{it}^d \in [\alpha_{it}^{d,l}, \alpha_{it}^{d,u}] $, then the current optimal solution of model~(\ref{model:eco}) is still optimal. Furthermore, let $ \mathcal{A} $ be the active set of the equality constraints, the LMP $ \lambda_t $ is locally a linear function of $ \alpha_{it}^c, \alpha_{it}^d \in \mathcal{A} $. 
\end{theorem}

From Theorem~\ref{thm:sensitivity}, we can see that the BESSs owner can adjust the bidding in an allowable range without impacting the optimal solution by sensitivity analysis. 
The BESSs owner could then solve a self-scheduling model (\ref{model:bat}) to decide when to charge and when to discharge. 
Then the BESSs owner wishes to reduce the LMP on charging intervals and increase the LMP on discharging intervals. 
Thus the BESSs owner could use the locally optimal biddings $ \alpha_{it}^c $, $ \alpha_{it}^d $ to bid in the ERCOT market. 
However, when the optimal solution changes, the sensitivity analysis method does not apply and we need to study the interaction between the ISO and the BESSs owner. 
Therefore, in the next subsection, we will establish the convergence of the AUS Algorithm \ref{alg:aus} and study the optimality of model (\ref{model:eco}) - (\ref{model:bat}).

\subsection{Convergence of Alternative Updating Scheme}\label{subsec:convergence}

In this subsection, we establish the convergence of the AUS. 

To avoid monopolization and manipulation of price, ISOs usually do market power mitigation to address non-competitive behaviors. Therefore, we use the following assumption to guarantee the existence of solution to model~(\ref{model:eco}) even when the BESS owner do not bid into the market. 
\begin{assumption}
	Model~(\ref{model:eco}) is feasible even when the BESS owner does not commit to the market.
\end{assumption}

As a result, we have the following lemma.

\begin{lemma}
	\label{lem:critical}
	There exists critical prices $ \underline{\lambda} $ and $ \overline{\lambda} $ such that $ \underline{\lambda} \le \lambda^* \le \overline{\lambda} $ no matter whether the BESS owner commits to the market or not. 
\end{lemma}

\begin{proof}
	Let $ (p_{it}^{c},p_{it}^{d},p_{mt}) $ be the objective solution to model~(\ref{model:eco}), then it satisfies the KKT conditions:
	\begin{align*}
		&\alpha_{it}^c - \mu_t - \sum_{l \in L}SF_{lb(i)}\pi^{+}_{l}+\sum_{l \in L}SF_{lb(i)}\pi^{-}_{l} \nonumber\\
		&+ \iota^{c+}_t-\iota^{c-}_t = 0, \forall t \in \Tt, \forall i \in I;\\
		&\alpha_{it}^d + \mu_t + \sum_{l \in L}SF_{lb(i)}\pi^{+}_{l}-\sum_{l \in L}SF_{lb(i)}\pi^{-}_{l} \nonumber\\
		&+ \iota^{d+}_t-\iota^{d-}_t = 0, \forall t \in \Tt, \forall i \in I;\\
		&\beta_{m} + \mu_t + \sum_{l \in L}SF_{lb(m)}\pi^{+}_{l}-\sum_{l \in L}SF_{lb(m)}\pi^{-}_{l} \nonumber\\
		&+ \upsilon^{c+}_t-\upsilon^{c-}_t = 0, \forall t \in \Tt, \forall m \in M;\\
		&\iota^{c+}_t(p_{it}^c-P^{cu}_{it}) = 0, \forall t \in \Tt;\\
		&\iota^{c-}_t (p_{it}^c-P^{cl}_{it}) = 0, \forall t \in \Tt;\\
		&\iota^{d+}_t(p_{it}^d-P^{du}_{it}) = 0, \forall t \in \Tt;\\
		&\iota^{d-}_t (p_{it}^d-P^{dl}_{it}) = 0, \forall t \in \Tt;\\
		&\upsilon^{+}_t(p_{mt}-P^u_m) = 0, \forall t \in \Tt;\\
		&\upsilon^{-}_t(p_{mt}-P^l_m) = 0, \forall t \in \Tt;\\
		&\pi^{\pm}\geq 0, \iota^{c\pm}_t\geq 0, \iota^{d\pm}_t\geq 0, \upsilon^{\pm}_t\geq 0, \forall t \in \Tt;\\
		&\mbox{(\ref{con:BESS-SS-balance})-(\ref{con:box-pd})}.
	\end{align*}
	Where $ b(i) $ denotes the bus of BESS $ i $ and $ b(m) $ denotes the bus of generator $ m $. 
	Recall the definition of LMP in (\ref{def:lmp}), this completes the proof. 

\end{proof}

As a result, the reasonable bids of BESS owner is a compact set. Combining Lemma~\ref{lem:critical} and Theorem~\ref{thm:sensitivity}, we have, 
\begin{theorem}\label{thm:nash}
	There exists finite Nash equilibrium of the problem between the BESS owner and the ISO.
\end{theorem}

However, ISOs are required to clear the market in limited time. To achieve this, ISOs need to accelerate the convergence of the bidding process and to avoid the oscillation between Nash equilibria. In the following, we study the conditions under which the bidding process converges to a unique Nash equilibrium.
Let 
\begin{align}
	&f(\alpha^c,\alpha^d) = \sum_{i \in I}(\sum_{t=1}^{T}\alpha^c_{it}p_{it}^c+\sum_{t=1}^{T}\alpha^d_{it}p_{it}^d ) +\sum_{m \in M}\sum_{t=1}^{T}\beta_m p_{mt};\\
	&g(\lambda) = \sum_{t=1}^{T}\lambda_t (p_{it}^c-p_{it}^d)= \sum_{t=1}^{T}\lambda_t p_{it}^c- \sum_{t=1}^{T}\lambda_t p_{it}^d.
\end{align}

We first define the following regularity conditions for the bidding strategy $ \mathcal{S} $: 

\begin{condition}\label{condition:mono}
	$ \partial\lambda_t^* = \partial_{l_t}\lambda_t(f^*(\mathcal{S}(p^{c,*}(l_t),p^{d,*}(l_t),l_t))) \ge 0 $, $ \forall t\in \Tt $. 
\end{condition}
We remark that Condition~\ref{condition:mono} excludes the bidding strategies that are obviously irrational for the BESSs owners, as otherwise there will be arbitrage opportunities by constructing a hedge portfolio. For example, a simple corollary of Condition~\ref{condition:mono} is that $ \alpha^c \le 0, \alpha^d \ge 0 $. 
Moreover, from the sensitivity analysis, $ \lambda_t $ is monotonic with respect to $ \alpha_{it}^c $ and $ \alpha_{it}^d $. The BESS can only bid into the market if $ \alpha_{it}^d \le \lambda_t $ and $ |\alpha_{it}^c| \ge \lambda_t $. Thus the market mechanism causes the following result. 

\begin{lemma}\label{lem:mono}
	Assume the BESS $ i $ is the only unit at the studied bus. Assume $ S^l c_{i} + \frac{\kappa^d c_i}{\eta^{d}}\Delta \leq e_{it} \leq S^u c_{i} - \kappa^cc_i\eta^{c}\Delta $. If $ \mathcal{S} $ is univalued and satisfies Condition~\ref{condition:mono}, then ISO and BESS $ i $ has a unique Nash equilibrium which is Pareto optimal if the BESS $ i $ tries and is able to bid into the market. 
\end{lemma}

\begin{proof}
	For brevity, in the following, the symbol $ \forall t\in\mathcal{T} $ has been omitted. 
	We prove by contradiction. Assume there is a Pareto improvement. 
	Suppose $ f' \le f $ and $ g' < g $. 
	Because the BESS is the only unit at this bus and successfully bids into the market, according to market characteristics and Condition~\ref{condition:mono}, 
	$ (\alpha^{c})' \le \alpha^c $ or $ (\alpha^{d})' \le \alpha^d $ and hence $ (\lambda' - \lambda) (p^c-p^d) \ge 0 $, which implies that $ g' \ge g $. 
	This contradicts the assumption. Therefore, it completes the proof and every optimal solution of model~(\ref{model:eco})-(\ref{model:bat}) is Pareto optimal.
\end{proof}

If the BESSs owner tries to bid into the power market, the bidding prices cannot be higher than LMP for discharging and lower than LMP for charging. Therefore, from Theorem~\ref{thm:nash} and Lemma~\ref{lem:mono}, we have
\begin{theorem}
	If the BESS is the only unit at the studied bus trying to bid into the market, then for any univalued strategy function $ \mathcal{S} $, satisfying Condition~\ref{condition:mono}, Algorithm~\ref{alg:aus} converges to the optimal solution of model~(\ref{model:eco})-(\ref{model:bat}) if $ \forall t \in \Tt $, such that $ S^l c_{i} + \frac{\kappa^d c_i}{\eta^{d}}\Delta \leq e_{it} \leq S^u c_{i} - \kappa^cc_i\eta^{c}\Delta $. 
\end{theorem}

\section{Experimental Results} \label{sec:exp}
In this section, we first describe the data set used in our experiments, then run the simulations to validate the efficacy of our new model and algorithms under different scenarios. Finally, we report all the simulation results. All numerical experiments are conducted on an AMD Ryzen™ Threadripper™ 3970X Processor with 64 threads on 3.7GHz and 256 GB memory. 
We use the ERCOT transmission network data from Breakthrough Energy~\cite{xu2020us} and historical loads and generation data from ERCOT website. All data are publicly available. 
We first use heuristic method in subsection~\ref{subsec:heuristic} to generate 100 BESSs' configurations, then we run Algorithm~\ref{alg:bayes} to solve the BESSIP model. 2880 hours of computational resources are allocated to Algorithm~\ref{alg:bayes}. 

\begin{figure}[htbp]
	\centering
	\includegraphics[width=\linewidth]{"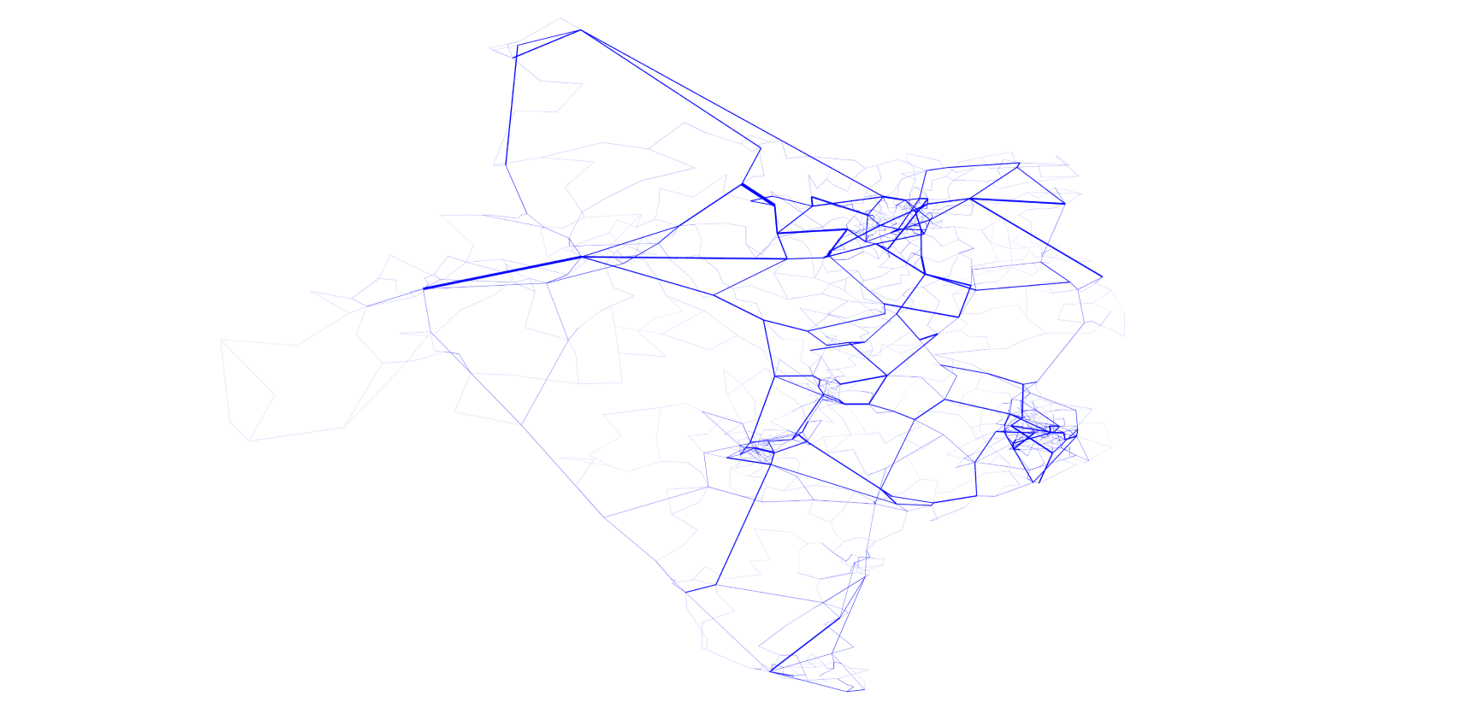"}
	\caption{Expectation of ERCOT Transmission Network Congestion Frequency}
	\label{fig:congestion}
\end{figure}

In order to visualize the congestion pattern of ERCOT, we show the average congestion frequency of ERCOT transmission network over a ten-year simulation in Figure~\ref{fig:congestion}, deeper color represents higher congestion frequency. In the experiment, the ERCOT transmission network comprises 2000 nodes. From the figure, we observe the high future congestion possibility in metropolitan areas in Texas, which appears as deep color clusters in the figure. In the meanwhile, the impact of wind farms in western Texas on the congestion frequency can also be observed from the figure as east-west tunnel with deep color, especially on east-west arterial transmission lines. The transmission congestion will be limiting the power flow from the wind farms to the load centers in the eastern part of Texas.
\begin{figure}[htbp]
	\centering
	\includegraphics[width=\linewidth]{"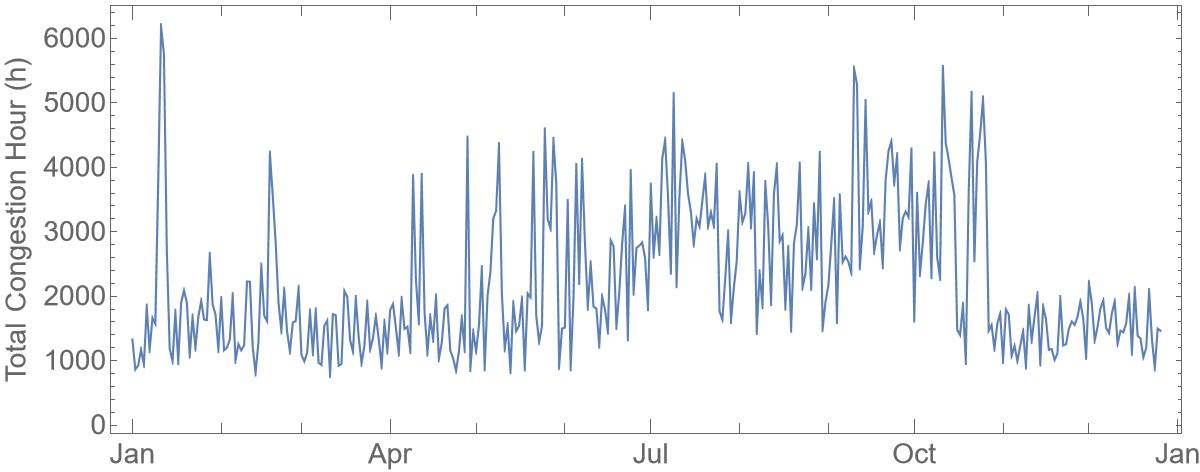"}
	\caption{Daily Total Congestion Hours of Simulated ERCOT Transmission Network}
	\label{fig:conghour}
\end{figure}
Figure~\ref{fig:conghour} shows the yearly distribution of average congestion frequency of ERCOT transmission network over ten-year simulation.
As shown in the figure, the congestion frequency is very high in the summer months, since the demand is usually high in summer months. 
Therefore, we conducted experiments of ERCOT simulation on 25\% days with highest $ S_{cong} $, which reduces the computational costs to $ \frac{1}{4} $. 87\% of the concentrated experiments give the same optimal BESSs locations as the original experiments, while among these experiments, the mean percentage error of the optimal BESS size is less than 15\%. 
This shows we can efficiently reduce the computational costs by only conducting simulation on days with the highest $ S_{cong} $ while not sacrificing too much solution quality. 

\begin{figure}[htbp]
	\centering
	\includegraphics[width=\linewidth]{"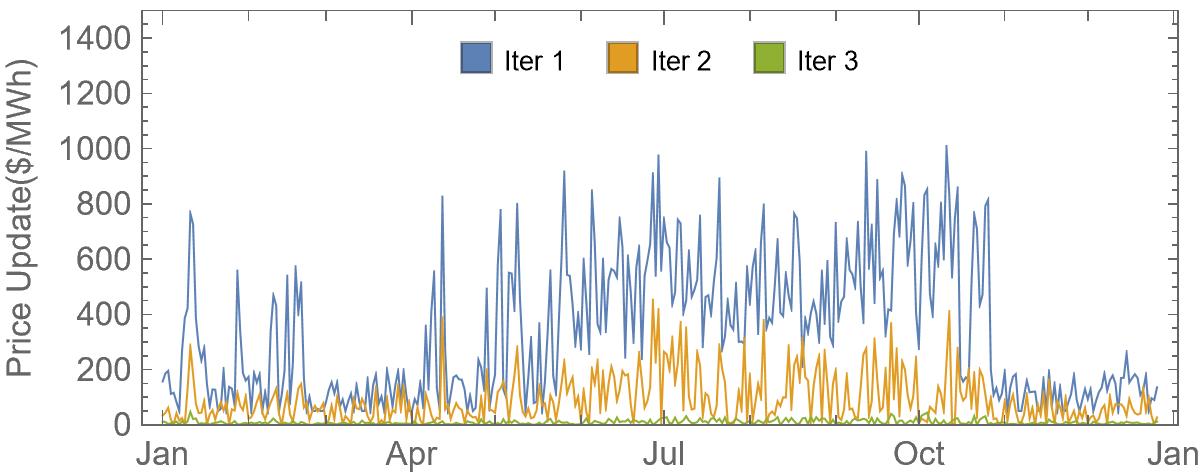"}
	\caption{$L_2$ Norm of the Daily Electricity Price Differences Between Each Iteration of the Algorithm}
	\label{fig:iter}
\end{figure}

Figure~\ref{fig:iter} shows the $L_2$ norm of the daily electricity price differences between each iteration of Algorithm~\ref{alg:aus} in Level 3, i.e. $ \norm{\lambda^k-\lambda^{k-1}}_{2} $ for $ k = 1,2,3 $. The figure shows that the $L_2$ norm of the daily electricity price differences is negligible after three iterations of the algorithm, which shows the convergence of the algorithm.

\begin{figure}[htbp]
	\centering
	\includegraphics[width=\linewidth]{"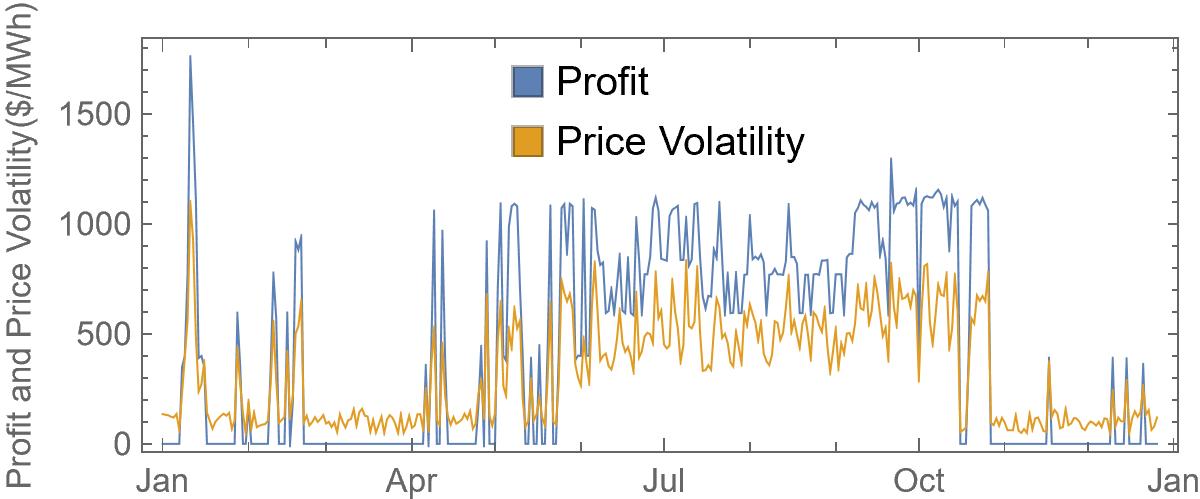"}
	\caption{Daily Volatility of the Electricity Price and the Daily Profit of the BESSs Per Unit of Installed Capacity (MWh)}
	\label{fig:pp}
\end{figure}

Figure~\ref{fig:pp} shows the daily volatility of the electricity price and the daily profit of the BESSs per unit of installed capacity (MWh). The figure shows a high correlation between the daily volatility of the electricity price and the daily profit of BESSs, which validates the fact that BESSs owners have higher arbitrage profits with higher price volatility. 
\begin{table}[!htbp]
	\centering
	\caption{Comparison of Different Models}
	\resizebox{\columnwidth}{!}{
		\begin{tabular}{lll}
		\toprule
		\multicolumn{3}{c}{Expected ROI (k\$) W/O Heuristics} \\
		BESS Impact & W/ BESS Impact & W/O BESS Impact \\
		Random Search & 25.631 & 31.849 \\
		Bayesian Optimization & 237.842 & 269.693 \\
		\midrule
		\multicolumn{3}{c}{Expected ROI (k\$) W/ Heuristics} \\
		BESS Impact & W/ BESS Impact & W/O BESS Impact \\
		Random Search & 136.351 & 152.464 \\
		Bayesian Optimization & 309.734 & 390.511 \\
		\bottomrule
	\end{tabular}%
	}
	\label{tab:comp}%
\end{table}%

Table~\ref{tab:comp} shows the comparison of the return on investment of BESSs with Bayesian optimization and random search under different assumptions with same computational resources. 
The table shows that the return on investment of BESSs is much higher with Bayesian optimization than the return with random search. This is because Bayesian optimization can find the optimal BESSs location in a much shorter time than a random search. From our experiments, Bayesian optimization obtains the same solution with 9\% computational resources of a random search, or 6\% computational resources of a grid search. The Bayesian based approach is at least 16 times faster than the state-of-the-art grid search. 
Another interesting observation is that the heuristic-based model can find better BESSs' configurations than the model without heuristics. This is because the score of congestion criterion can reflect the potential arbitrageable space of the BESSs and reduce the search space of the BESSs location significantly.
From Table~\ref{tab:comp}, we can also observe that the return on investment of BESSs is overestimated 21\% when the BESSs' impact on the transmission network is not considered. This is because the BESSs' impact on transmission network can reduce transmission congestion, which can further reduce the electricity price volatility and hence reduce the potential arbitrage profit of the BESSs.
\begin{table}[htbp]
	\centering
	\caption{Comparison of Solutions and Computational Costs of Scenarios}	
	\resizebox{\columnwidth}{!}{
		\begin{tabular}{lll}
			\toprule
			Model & Deviation & Computational Cost (s) \\
			BESSIP (Model (\ref{model:eco})-(\ref{model:bat})) & NA    & 3021 \\
			Zero Fixed Cost (Model (\ref{model:S})) & 3.2\% & 2996 \\
			Complementarity Constraints (Constraint (\ref{con:complimentary})) & 3.5\% & 19159 \\
			Zero Fixed Cost + Complementarity Constraints & 5.4\% & 18925 \\
			\bottomrule
		\end{tabular}%
}
	\label{tab:comp-assump}%
\end{table}%

Table~\ref{tab:comp-assump} presents a comparison of different scenarios' solutions and computational costs. The four scenarios are BESSIP and variants of BESSIP with zero fixed costs, complementarity constraints. Deviation indicates the average percentage difference between the scenario's solution and the optimal solution of BESSIP. 
The computational costs of BESSIP variants with the complimentarity constraints are higher than BESSIP, as the complimentarity constraints are nonconvex bilinear constraints. BESSIP with zero fixed costs slightly differs from BESSIP, indicating that fixed costs can be neglected in long-term planning problems in some practical cases.

\section{Conclusion} \label{sec:conclusion}
In this paper, we considered the strategic BESS investment problem and proposed a three-phase approach for it. We presented effective heuristics and algorithms to solve the sub-problems in each phase. 
Our experiments illustrated that the optimal solutions obtained from our proposed approach provide valuable insights into the ERCOT electrical grid's structure, particularly regarding the changes in congestion patterns resulting from the deployment of large-scale BESSs. 
Through our experiments, we demonstrated that the proposed approach can significantly reduce the computational burden of the BESS investment problem.
We established the convergence of the proposed approach and showed the Pareto optimality of the problem.

\section*{Acknowledgement}
We would like to thank Dr. Siddharth Bhela at Siemens Technolgoy for insightful discussions.

\bibliographystyle{IEEEtran}
\bibliography{./references}

\end{document}